\documentclass[12pt, reqno]{amsart}

\usepackage{amsmath, amsthm, amscd, amsfonts, amssymb, graphicx, color}
\usepackage[bookmarksnumbered, colorlinks, plainpages]{hyperref}% this package creats hyperlinks.

\newtheorem{theorem}{Theorem}[section]
\newtheorem{lemma}[theorem]{Lemma}

\theoremstyle{definition}

\theoremstyle{remark}
\newtheorem{remark}[theorem]{Remark}
\begin{document}
%-----------------------------------------------PAPER / ARTICLE FORMATINGS---------------------------------------------------
\setcounter{page}{1}

\title[Inequalities for the Psi and $k$-Gamma Functions]{Inequalities for the Psi and $k$-Gamma Functions}

\author[Kwara Nantomah]{Kwara Nantomah$^{*}$$^1$}

\address{$^{1}$ Department of Mathematics, University for Development Studies, Navrongo Campus, P. O. Box 24, Navrongo, UE/R, Ghana. }
\email{\textcolor[rgb]{0.00,0.00,0.84}{mykwarasoft@yahoo.com, knantomah@uds.edu.gh}}

%\address{$^{2}$ Department of Mathematics, Kwame Nrumah University of Science and Technology, Kumasi, Ghana. }
%\email{\textcolor[rgb]{0.00,0.00,0.84}{eprempeh.cos@knust.edu.gh}}

%\dedicatory{This paper is dedicated to Professor ABCD}



\date{09 January 2016.
\newline \indent $^{*}$ Corresponding author}

\begin{abstract}
In this paper, the authors establish some inequalities involving the Psi and $k$-Gamma functions. The procedure utilizes some monotonicity properties of some functions associated with the Psi and $k$-Gamma functions.
\end{abstract} \maketitle

%--------------------------------------------------------BODY OF ARTICLE --------------------------------------------------------
\section{Introduction}

\noindent
The well-known classical Gamma function, $\Gamma(t)$ is usually defined for $t>0$ by
\begin{equation*}\label{eqn:gamma}
\Gamma(t)=\int_0^\infty e^{-x}x^{t-1}\,dx.
\end{equation*}

\noindent
The $p$-analogue of the Gamma function is defined (see also \cite{Krasniqi-Mansour-Shabani-2010}, \cite{Krasniqi-Shabani-2010}) for $t>0$  and $p\in N$ by 
\begin{equation*}\label{eqn:p-gamma}
\Gamma_p(t)=\frac{p!p^t}{t(t+1) \dots (t+p)}=\frac{p^t}{t(1+\frac{t}{1}) \dots (1+\frac{t}{p})}.
\end{equation*}

\noindent
Also, the  $q$-analogue of the Gamma function is defined (see \cite{Mansour-2008}) for  $t>0$  and $q\in(0,1)$ by
\begin{equation*}\label{eqn:q-gamma}
\Gamma_q(t)=(1-q)^{1-t}\prod_{n=1}^{\infty}\frac{1-q^n}{1-q^{n+t}}=(1-q)^{1-t}\prod_{n=0}^{\infty}\frac{1-q^{n+1}}{1-q^{n+t}}.
\end{equation*}

\noindent
Similarly, the $k$-analogue or the $k$-Gamma function is  defined (see \cite{Diaz-Pariguan-2007}) for $t>0$  and $k>0$ by 
\begin{equation*}\label{eqn:k-gamma}
\Gamma_k(t)=\int_0^\infty e^{-\frac{x^k}{k}}x^{t-1}\,dx.
\end{equation*}

\noindent
The psi function, $\psi(t)$ also known in literature as the digamma function is defined for $t>0$ as the logarithmic derivative of the gamma function. That is,
\begin{equation*}\label{eqn:digamma}
\psi(t)=\frac{d}{dt}\ln(\Gamma(t))=\frac{\Gamma'(t)}{\Gamma(t)}.
\end{equation*}

\noindent
The $p$-analogue, $q$-analogue and $k$-analogue of the psi function are equivalently defined for $t>0$ as follows.
\begin{equation*}
\psi_p(t)=\frac{d}{dt}\ln(\Gamma_p(t))=\frac{\Gamma'_p(t)}{\Gamma_p(t)},  \quad   \psi_q(t)=\frac{d}{dt}\ln(\Gamma_q(t))=\frac{\Gamma'_q(t)}{\Gamma_q(t)}  \quad \text{and}
\end{equation*}

\begin{equation*}
\psi_k(t)=\frac{d}{dt}\ln(\Gamma_k(t))=\frac{\Gamma'_k(t)}{\Gamma_k(t)}.
\end{equation*}

\noindent
The following series representations for the functions $\psi(t)$ and $\psi_k(t)$ are valid and are well-known in literature.
\begin{align}
\psi(t)&=-\gamma - \frac{1}{t} + \sum_{n=1}^{\infty}\frac{t}{n(n+t)}  \label{eqn:series-psi}\\
\psi_k(t)&=\frac{\ln k-\gamma}{k}-\frac{1}{t}+\sum_{n=1}^{\infty} \frac{t}{nk(nk+t)}   \label{eqn:k-psi2}
\end{align}
where $\gamma$ denotes the Euler-Mascheroni's constant.\\

\noindent
The polygamma functions, $\psi^{(m)}(t)$  are defined for $t>0$ and $m\in N$ as the $m$-th derivative of the psi function. That is,
\begin{equation*}\label{eqn:polygamma} 
\psi^{(m)}(t) = \frac{d^{m}}{dt^{m}}\psi(t) = \frac{d^{m+1}}{dt^{m+1}}\ln(\Gamma(t)).
\end{equation*}
where $\psi^{(0)}(t)\equiv \psi(t)$. They also exhibit the series representation shown below.

\begin{equation}\label{eqn:series-polygamma}
\psi^{(m)}(t) = (-1)^{m+1}m! \sum_{n=0}^{\infty}\frac{1}{(n+t)^{m+1}}
\end{equation}
Consequently, the following representations are trivially obtained from ~(\ref{eqn:series-polygamma}).

\begin{align}
\psi'(t) &=  \sum_{n=0}^{\infty}\frac{1}{(n+t)^{2}}  \label{eqn:series-digamma-prime}\\
\psi^{(m+1)}(t)& = (-1)^{m+2}(m+1)! \sum_{n=0}^{\infty}\frac{1}{(n+t)^{m+2}}  \label{eqn:series-polygamma-prime}
\end{align}

\noindent
By using basic analyses, the purpose of this paper is to establish  some inequalities for Psi and $k$-Gamma functions. We present our results in the following sections.

%--------------------------------------------------- MAIN RESULTS -------------------------------------------------------------

%%%%%%%%%%%%%%%%%%%%%%%%%%%%%%%%%%%%%%%%%%%%%%%%%%%%%
%                                                                                   Section One
%%%%%%%%%%%%%%%%%%%%%%%%%%%%%%%%%%%%%%%%%%%%%%%%%%%%%

\section{ Some inequalities for the Psi function}
\noindent
This section is devoted to some inequalities associated with the Psi function. We proceed as follows.

\begin{lemma}\label{lem:digamma-increasing}
Let  $0<s\leq t$, then the following statement holds true.
\begin{equation}
\psi(s) \leq \psi(t).   \label{eqn:digamma-increasing}
\end{equation}
\end{lemma}

\begin{proof}
From ~(\ref{eqn:series-psi}), we have the following (See also \cite{Shabani-2008a}).
\begin{align*}
\psi(s)-\psi(t) &= (s-1)\sum_{n=0}^{\infty}\frac{1}{(n+1)(n+s)}-(t-1)\sum_{n=0}^{\infty}\frac{1}{(n+1)(n+t)}\\
&= \sum_{n=0}^{\infty}\frac{1}{(n+1)}\left(\frac{s-1}{n+s}-\frac{t-1}{n+t}\right)\\
&= \sum_{n=0}^{\infty}\frac{(s-t)}{(n+s)(n+t)}\leq0 .
\end{align*}
\end{proof}

\begin{lemma}\label{lem:digamma-prime-increasing}
Let  $0<s\leq t$, then the following statement holds true.
\begin{equation}
\psi'(s) \geq \psi'(t).   \label{eqn:digamma-prime-increasing}
\end{equation}
\end{lemma}

\begin{proof}
From ~(\ref{eqn:series-digamma-prime}), we have the following.
\begin{align*}
\psi'(s)-\psi'(t) &= \sum_{n=0}^{\infty}\frac{1}{(n+s)^2} - \sum_{n=0}^{\infty}\frac{1}{(n+t)^2}\\
&= \sum_{n=0}^{\infty}\left[ \frac{1}{(n+s)^2} - \frac{1}{(n+t)^2}  \right]\\
&= \sum_{n=0}^{\infty}\frac{2n(t-s)+(t^2-s^2)}{(n+s)^{2}(n+t)^{2}}\geq0 .
\end{align*}
\end{proof}

% ------------------------------------------------------ theorem 1 --------------------------------------------------------------
\begin{theorem}\label{thm:psi-funct}
Define a function $U$ by 
\begin{equation*}\label{eqn:psi-funct} 
U(t)=\frac{\left[ \psi(a+bt)\right] ^{\alpha}}{\left[ \psi(c+dt)\right] ^{\beta}},  \quad t\in [0,\infty)
\end{equation*}
where  $a$, $b$,  $c$, $d$, $\alpha$, $\beta$ are  positive real numbers such that $a\leq c$, $b \leq d$,  $\beta d \leq \alpha b$, $a+bt \leq c+dt$, $\psi(a+bt)>0$ and $\psi(c+dt)>0$.
Then $U$ is non-decreasing on $t\in[0,\infty)$ and the inequalities
\begin{equation}\label{eqn:psi-funct-ineq}
\frac{\left[ \psi(a)\right] ^{\alpha}}{\left[ \psi(c)\right] ^{\beta}} \leq
\frac{\left[ \psi(a+bt)\right] ^{\alpha}}{\left[ \psi(c+dt)\right] ^{\beta}} \leq
\frac{\left[ \psi(a+b)\right] ^{\alpha}}{\left[ \psi(c+d)\right] ^{\beta}}
\end{equation}
hold true for  $t\in[0,1]$.
\end{theorem}
%--------------------------------------------  proof of  theorem 1 ------------------------------------------------
\begin{proof}
Let $\mu(t)=\ln U(t)$ for every $t\in[0,\infty)$. Then,
\begin{align*}
\mu &=\ln \frac{\left[ \psi(a+bt)\right] ^{\alpha}}{\left[ \psi(c+dt)\right] ^{\beta}}
        = \alpha \ln \psi(a+bt) - \beta \ln \psi(c+dt)   \\
\intertext{and}
\mu'(t)&= \alpha b \frac{\psi'(a+bt)}{\psi(a+bt)} - \beta d \frac{\psi'(c+dt)}{\psi(c+dt)} \\
      &= \frac{\alpha b \psi'(a+bt)\psi(c+dt) - \beta d \psi'(c+dt)\psi(a+bt) }{\psi(a+bt) \psi(c+dt) }. \quad
% \text{(by Lemma~\ref{lem:psi-p-psi2})}
\end{align*}
Since $0<a+bt \leq c+dt$, then by Lemmas \ref{lem:digamma-increasing} and \ref{lem:digamma-prime-increasing} we have,\\
$\psi(a+bt)\leq \psi(c+dt)$ and $\psi'(a+bt)\geq \psi'(c+dt)$. Then that implies;\\
$\psi(c+dt)\psi'(a+bt) \geq \psi(c+dt)\psi'(c+dt) \geq \psi(a+bt)\psi'(c+dt)$. Further, $\alpha b \geq \beta d$ implies;\\
$\alpha b \psi(c+dt)\psi'(a+bt) \geq \alpha b \psi(a+bt)\psi'(c+dt)  \geq \beta d \psi(a+bt)\psi'(c+dt)$. Hence,\\
$\alpha b \psi(c+dt)\psi'(a+bt) - \beta d \psi(a+bt)\psi'(c+dt)\geq0$. Therefore  $\mu'(t)\geq0$.\\
That implies $\mu$ as well as $U$ are non-decreasing on $t\in[0,\infty)$ and  for $t\in[0,1]$ we have, 
\begin{equation*}
U(0) \leq U(t) \leq U(1)   
\end{equation*}
resulting to inequalities ~(\ref{eqn:psi-funct-ineq}).
\end{proof}

\begin{remark}
If in particular, $a=b=c=d=1$ then we obtain
\begin{equation*}
(-\gamma)^{\alpha - \beta} \leq \left[ \psi(1+t) \right]^{\alpha - \beta} \leq (1-\gamma)^{\alpha - \beta}.
\end{equation*}
\end{remark}

\begin{remark}
For $t\in(1,\infty)$,  we have \,$U(t)\geq U(1)$\, yielding
\begin{equation*}
\frac{\left[ \psi(a+bt)\right] ^{\alpha}}{\left[ \psi(c+dt)\right] ^{\beta}} \geq
\frac{\left[ \psi(a+b)\right] ^{\alpha}}{\left[ \psi(c+d)\right] ^{\beta}}.
\end{equation*}
\end{remark}

\begin{remark}
Results similar to Theorem \ref{thm:psi-funct} can also be found  in \cite{Nantomah-Iddrisu-2014} for the $k$-analogue of the psi function.\\
\end{remark}

%================================================================
% -----------Lemma 1-------------------
\begin{lemma}\label{lem:Polygamma-Turan-Ineq-Particular}
Let $m$ be a positive odd integer and $t>0$. Then 
\begin{equation*}
\psi^{(m)}(t)\psi^{(m+2)}(t) - \left[ \psi^{(m+1)}(t) \right]^2 \geq0.
\end{equation*}
\end{lemma}

\begin{proof}
It was established in \cite[Theorem 2.1]{Laforgia-Natalini-2006-JIPAM} that   $\psi^{(m)}(t)\psi^{(n)}(t) \geq  \left[ \psi^{(\frac{m+n}{2})}(t) \right]^2 $, where $\frac{m+n}{2}$ is an integer and $t>0$. Let $n=m+2$ then,  the proof is complete.
\end{proof}

% -----------Lemma 2-------------------

\begin{lemma}\label{lem:Polygamma-Ratio-Increasing}
Let $m$ be a positive odd integer. Then for $0<s\leq t$, we have 
\begin{equation*}
\frac{\psi^{(m+1)}(s)}{\psi^{(m)}(s)} \leq \frac{\psi^{(m+1)}(t)}{\psi^{(m)}(t)}.  
\end{equation*}
\end{lemma}

\begin{proof}
Let $Q(y)=\frac{\psi^{(m+1)}(y)}{\psi^{(m)}(y)}$ , where $m$ is a positive odd integer and $y>0$. Then,
\begin{equation*}
Q'(y)=\frac{\psi^{(m)}(y)\psi^{(m+2)}(y)-[ \psi^{(m+1)}(y) ]^2 }{[ \psi^{(m)}(y)]^2} 
\end{equation*}
and by Lemma \ref{lem:Polygamma-Turan-Ineq-Particular}, $Q'(y)\geq0$. Thus, $Q$ is increasing. Then for $0<s\leq t$, we obtain
\begin{equation*}
\frac{\psi^{(m+1)}(s)}{\psi^{(m)}(s)} \leq \frac{\psi^{(m+1)}(t)}{\psi^{(m)}(t)}   
\end{equation*}
concluding the proof.
\end{proof}

% ------------------------------------------------------ theorem 2 --------------------------------------------------------------
\begin{theorem}\label{thm:polygamma-funct-1}
For a positive odd integer $m$, define a function $V$ by 
\begin{equation*}\label{eqn:polygamma-funct-1} 
V(t)=\frac{ \psi^{(m)}(\alpha+t)}{\psi^{(m)}(\beta+t)},  \quad t\in [0,\infty)
\end{equation*}
where $0<\alpha \leq \beta$ are  real numbers. Then $V$ is decreasing on $t\in[0,\infty)$ and the inequalities
\begin{equation}\label{eqn:polygamma-funct-1-ineq}
\frac{ \psi^{(m)}(\alpha)}{\psi^{(m)}(\beta)} \geq
\frac{ \psi^{(m)}(\alpha+t)}{\psi^{(m)}(\beta+t)} \geq
\frac{ \psi^{(m)}(\alpha+1)}{\psi^{(m)}(\beta+1)}
\end{equation}
are valid for $t\in[0,1]$.\\
\end{theorem}

%--------------------------------------------  proof of  theorem 2 ------------------------------------------------
\begin{proof}
Let $f(t)=\ln V(t)$ for $t\in[0,\infty)$. That is,
\begin{align*}
f(t) &=\ln \frac{ \psi^{(m)}(\alpha+t)}{\psi^{(m)}(\beta+t)}
        =  \ln \psi^{(m)}(\alpha+t) - \ln \psi^{(m)}(\beta+t). \quad \text{Then,}\\
f'(t)&=  \frac{\psi^{(m+1)}(\alpha+t)}{\psi^{(m)}(\alpha+t)} -  \frac{\psi^{(m+1)}(\beta+t)}{\psi^{(m)}(\beta+t)} \\
   &\leq 0 
\end{align*}
as a result of Lemma \ref{lem:Polygamma-Ratio-Increasing}. Thus $f$ and for that matter $V$ are decreasing on $t\in[0,\infty)$ and  for $t\in[0,1]$ we have, 
\begin{equation*}
V(0) \geq V(t) \geq V(1)  
\end{equation*}
resulting to inequalities ~(\ref{eqn:polygamma-funct-1-ineq}).
\end{proof}

\begin{remark}
If  $t\in(1,\infty)$, then  we have \,$V(t) < V(1)$\, yielding 
\end{remark}
\begin{equation*}
\frac{ \psi^{(m)}(\alpha+t)}{\psi^{(m)}(\beta+t)} <
\frac{ \psi^{(m)}(\alpha+1)}{\psi^{(m)}(\beta+1)}.
\end{equation*}
%================================================================

%%%%%%%%%%%%%%%%%%%%%%%%%%%%%%%%%%%%%%%%%%%%%%%%%%%%%
%                                                                                   Section Two
%%%%%%%%%%%%%%%%%%%%%%%%%%%%%%%%%%%%%%%%%%%%%%%%%%%%%

\section{ Some inequalities for the $k$-Gamma Function}

\noindent
This section is dedicated to some inequalities associated with the $k$-Gamma function.

\noindent
In 2010, Krasniqi and Shabani \cite{Krasniqi-Shabani-2010} proved that,
\begin{equation*}\label{eqn:p-ineq1}
\frac{p^{-t}e^{-\gamma t}\Gamma(\alpha)}{\Gamma_p(\alpha)}< 
\frac{\Gamma(\alpha+t)}{\Gamma_p(\alpha+t)}<
\frac{p^{1-t}e^{\gamma(1-t)} \Gamma(\alpha+1)}{\Gamma_p(\alpha+1)}
\end{equation*}
for $p\in N$, $t\in(0,1)$, where  $\alpha$ is a positive real number such that $\alpha+t>1$.\\

\noindent
Also in that same year, Krasniqi, Mansour and Shabani \cite{Krasniqi-Mansour-Shabani-2010} proved the following: 
\begin{equation*}\label{eqn:q-ineq1}
\frac{(1-q)^{t}e^{-\gamma t} \Gamma(\alpha)}{\Gamma_q(\alpha)}< 
\frac{\Gamma(\alpha+t)}{\Gamma_q(\alpha+t)}<
 \frac{(1-q)^{t-1}e^{\gamma(1-t)}\Gamma(\alpha+1)}{\Gamma_q(\alpha+1)}
\end{equation*}
for $q\in(0,1)$,  $t\in(0,1)$, where  $\alpha$ is a positive real number such that $\alpha+t>1$.\\

\noindent
In this section, our interest is to establish similar inequalities for the $k$-Gamma function. We also present  some new results involving products of certain ratios of the $k$-Gamma function. We proceed as follows.\\

\begin{lemma}\label{lem:psi-kpsi2}
Let $k\ge1$ and $\alpha>0$ such that $\alpha+t>0$. Then,
\begin{equation*}\label{eqn:psi-kpsi2} 
\gamma + \frac{\ln k -\gamma}{k}+\psi(\alpha+t) - \psi_k(\alpha+t) \geq0.
\end{equation*}
\end{lemma}

\begin{proof}
Using  equations ~(\ref{eqn:series-psi}) and ~(\ref{eqn:k-psi2}) we obtain,
\begin{equation*}
\gamma + \frac{\ln k -\gamma}{k}+\psi(t) - \psi_k(t) = 
t\left[ \sum_{n=1}^{\infty}\frac{1}{n(n+t)} - \sum_{n=1}^{\infty}\frac{1}{nk(nk+t)} \right]\geq0.
\end{equation*}
\end{proof}
Substituting $t$ by $\alpha+t$ completes the proof.\\

%----------------------------------------------------------- theorem 1 ----------------------------------------------------------
\begin{theorem}\label{thm:funct2}
Define a function $W$ for $k\geq1$ by 
\begin{equation*}\label{eqn:funct2} 
W(t)=\frac{k^{\frac{t}{k}}e^{t \left( \frac{k\gamma - \gamma}{k} \right)}\Gamma(\alpha+t)}{\Gamma_k(\alpha+t)}    ,    \quad t\in (0,\infty)
\end{equation*}
where  $\alpha$ is a positive real number. % such that $\alpha+t>0$. 
Then $W$ is increasing on $t\in(0,\infty)$ and for $t\in(0,1)$, the following inequalities are valid.
\end{theorem}

\begin{equation}\label{eqn:k-digamma}
\frac{k^{-\frac{t}{k}}e^{-t \left( \frac{k\gamma - \gamma}{k} \right)}\Gamma(\alpha)}{\Gamma_k(\alpha)}\leq 
\frac{\Gamma(\alpha+t)}{\Gamma_k(\alpha+t)}\leq
\frac{k^{\frac{1-t}{k}}e^{(1-t) \left( \frac{k\gamma - \gamma}{k} \right)}\Gamma(\alpha+1)}{\Gamma_k(\alpha+1)}.
\end{equation}

%----------------------------------------------------- proof of theorem 1 -----------------------------------------------------
\begin{proof}
Let $v(t)=\ln W(t)$ for every $t\in(0,\infty)$. Then,
\begin{align*}
v(t) &=\ln \frac{k^{\frac{t}{k}}e^{t \left( \frac{k\gamma - \gamma}{k} \right)}\Gamma(\alpha+t)}{\Gamma_k(\alpha+t)} \\
&=\frac{t}{k}\ln k + t\left( \frac{k\gamma - \gamma}{k} \right) + \ln \Gamma(\alpha+t)-\ln \Gamma_k(\alpha+t) \\
\intertext{Then,}
v'(t)&=\frac{\ln k}{k}+ \frac{k\gamma - \gamma}{k} + \psi(\alpha+t) -  \psi_k(\alpha+t) \\
&= \gamma + \frac{\ln k -\gamma}{k}+\psi(\alpha+t) - \psi_k(\alpha+t) \geq0. \quad  \text{(by Lemma~\ref{lem:psi-kpsi2})}.
\end{align*}
\end{proof}
\noindent
That implies $v$ is increasing on $t\in(0,\infty)$. Hence $W=e^{v(t)}$ is increasing on $t\in(0,\infty)$ and  for $t\in(0,1)$ we have, 
\begin{equation*}
W(0) \leq W(t) \leq W(1) 
\end{equation*}
resulting to inequalities ~(\ref{eqn:k-digamma}).\\

\begin{remark}
For  $t\in[1,\infty)$,  we have $W(1)\leq W(t)$  yielding
\begin{equation*}
\frac{k^{\frac{1-t}{k}}e^{(1-t) \left( \frac{k\gamma - \gamma}{k} \right)}\Gamma(\alpha+1)}{\Gamma_k(\alpha+1)} \leq
\frac{\Gamma(\alpha+t)}{\Gamma_k(\alpha+t)}.
\end{equation*}
\end{remark}

%---------------------------------------------------------------------------------
\begin{lemma}\label{lem:k-psi4}
Let $k>0$,  $s>0$, $t>0$ with $s\leq t$, then
\begin{equation}\label{eqn:k-psi4}
\psi_k(s)\leq \psi_k(t).
\end{equation}
\end{lemma}

\begin{proof}
From ~(\ref{eqn:k-psi2}), we have the following.
\begin{align*}
\psi_k(s)-\psi_k(t) &=\frac{1}{t}-\frac{1}{s}+ \sum_{n=1}^{\infty}\left(  \frac{1}{nk}-\frac{1}{s+nk} \right)-
\sum_{n=1}^{\infty}\left(  \frac{1}{nk}-\frac{1}{t+nk} \right)\\
&= \frac{s-t}{st}+ \sum_{n=1}^{\infty}\left(  \frac{1}{t+nk}-\frac{1}{s+nk} \right) \\
&=  \frac{s-t}{st}+\sum_{n=1}^{\infty}\frac{(s-t)}{(s+nk)(t+nk)}\leq0.
\end{align*}
\end{proof}

\begin{lemma}\label{lem:k-psi5}
Let $a, b,  \alpha_i,  \beta_i$, $i=1,...,n$ be real numbers such that  $at+\alpha_i>0$, $bt+\beta_i>0$. Then,   $at+\alpha_i \leq bt+\beta_i$ implies   $\psi_k(at+\alpha_i)\leq\psi_k(bt+\beta_i)$.
\end{lemma}
\begin{proof}
A direct consequence of Lemma \ref{lem:k-psi4}. 
\end{proof}

\begin{lemma}\label{lem:k-psi6}
Let $a, b,  \alpha_i,  \beta_i$, $i=1,...,n$ be real numbers such that $0<a \leq b$, $at+\alpha_i>0$, $bt+\beta_i>0$,   $at+\alpha_i \leq bt+\beta_i$ and $\psi_k(at+\alpha_i)>0$. Then,  $a\psi_k(at+\alpha_i)\leq b\psi_k(bt+\beta_i)$.
\end{lemma}
\begin{proof}
From Lemma \ref{lem:k-psi5}, we have $\psi_k(at+\alpha_i)\leq\psi_k(bt+\beta_i)$. This together with the fact that 
$0<a \leq b$  yields,\\
$a\psi_k(at+\alpha_i) \leq a\psi_k(bt+\beta_i)\leq b\psi_k(bt+\beta_i)$  concluding the proof.\\
\end{proof}

% ------------------------------------------------------ theorem 2 --------------------------------------------------------------

\begin{theorem}\label{thm:funct1}
Define a function $X$ for $k>0$ by 
\begin{equation*}\label{eqn:funct1} 
X(t)=\prod_{i=1}^{n} \frac{\Gamma_k(at+\alpha_i )}{\Gamma_k(bt+\beta_i)}, \quad t\in [0,\infty)
\end{equation*}
where  $a, b,  \alpha_i,  \beta_i$, $i=1,...,n$ are real numbers such that $0<a\leq b$,  $\alpha_i>0$, $\beta_i>0$, $at+\alpha_i>0$, $bt+\beta_i>0$,   $at+\alpha_i \leq bt+\beta_i$ and $\psi_k(at+\alpha_i)>0$. Then $X$ is decreasing and for  $t\in[0,1]$, the following inequalities hold true.
\end{theorem}

\begin{equation}\label{eqn:k1}
\prod_{i=1}^{n} \frac{\Gamma_k(a+\alpha_i )}{\Gamma_k(b+\beta_i)}\leq 
\prod_{i=1}^{n} \frac{\Gamma_k(at+\alpha_i )}{\Gamma_k(bt+\beta_i)}\leq
\prod_{i=1}^{n} \frac{\Gamma_k(\alpha_i )}{\Gamma_k(\beta_i)}.
\end{equation}

%--------------------------------------------  proof of  theorem 2 ------------------------------------------------
\begin{proof}
Let $u(t)=\ln X(t)$ for every $t\in[0,\infty)$. Then,
\begin{align*}
u(t) &=\ln  \prod_{i=1}^{n} \frac{\Gamma_k(at+\alpha_i )}{\Gamma_k(bt+\beta_i)} \\
&=\sum_{i=1}^{n}\left[  \ln \Gamma_k(at+\alpha_i )  - \ln \Gamma_k(bt+\beta_i) \right] \\
\intertext{Then,}
u'(t)&=\sum_{i=1}^{n}\left[   a\frac{\Gamma'_k(at+\alpha_i )}{\Gamma_k(at+\alpha_i )} - 
                                                 b\frac{\Gamma'_k(bt+\beta_i)}{\Gamma_k(bt+\beta_i)} \right] \\
&= \sum_{i=1}^{n}\left[ a \psi_k(at+\alpha_i) - b \psi_k(bt+\beta_i) \right]  \leq0. \quad 
 \text{(by Lemma~\ref{lem:k-psi6})}.
\end{align*}
\end{proof}
\noindent
That implies $u$ is decreasing on $t\in[0,\infty)$. Hence, $X=e^{u(t)}$ is decreasing for each $t\in[0,\infty)$. Then for $t\in[0,1]$ we have, 
\begin{equation*}
X(1)\leq X(t)\leq X(0) 
\end{equation*}
resulting to inequalities ~(\ref{eqn:k1}).\\

\begin{remark}
For  $t\in(1,\infty)$, we have $X(t)\leq X(1)$ yielding
\begin{equation*}
\prod_{i=1}^{n} \frac{\Gamma_k(at+\alpha_i )}{\Gamma_k(bt+\beta_i)} \leq
\prod_{i=1}^{n} \frac{\Gamma_k(a+\alpha_i )}{\Gamma_k(b+\beta_i)}.
\end{equation*}
\end{remark}

\begin{remark}
If $0<b\leq a$,  $at+\alpha_i \geq bt+\beta_i$ and $\psi_k(bt+\beta_i)>0$, then  for  $t\in[0,1]$ the inequalities ~(\ref{eqn:k1}) are reversed. \\
\end{remark}

\section{Concluding Remarks}
\noindent
We have discovered that Lemma 2.7 and Theorem 2.8 of the paper \cite{Nantomah-2014-GJMA} are erroneous. The errors had to do with the claim that
\begin{equation*}
\psi^{(m+1)}(\alpha+t)\psi^{(m)}(\beta+t) -   \psi^{(m+1)}(\beta+t)\psi^{(m)}(\alpha+t) \geq0
\end{equation*}
and consequently that, $V(t)=\frac{ \psi^{(m)}(\alpha+t)}{\psi^{(m)}(\beta+t)}$ \,  is non-decreasing, where $t\geq0$, $0<\alpha \leq \beta$ and $m$ a positive odd integer. As a result, the inequalities \cite[eqn. (9)]{Nantomah-2014-GJMA} resulting from these claims are false.
This paper is therefore a corrected version of the paper \cite{Nantomah-2014-GJMA}.

%-------------------------------------------------------------------------------------------------------------------------------------------------
%{\bf ACKNOWLEDGEMENTS.}
%This is a text of acknowledgements.

%----------------------------------------------------BIBLIOGRAGHY-------------------------------------------------------------
\bibliographystyle{plain}

%------------------------------------------------------------------------------------------------------------------------------------

\end{document}